\documentclass[final]{siamltex}

\usepackage{graphicx}
\usepackage{color}
\usepackage{bm}
\usepackage{amsfonts}
\usepackage{amsmath}
\usepackage{multirow}

\newcommand{\C}[1]{{\cal {#1}}}

\title{
Lebesgue Constants Arising in \\
a Class of Collocation Methods
\thanks{
May 6, 2015, revised September 12, 2015
The authors gratefully acknowledge support by
the Office of Naval Research under grants N00014-11-1-0068 and
N00014-15-1-2048, and by the National Science Foundation under
grants DMS 1522629 and CBET-1404767.
}}

\author{
William W. Hager\thanks{{\tt hager@ufl.edu},
http://people.clas.ufl.edu/hager/,
PO Box 118105,
Department of Mathematics,
University of Florida, Gainesville, FL 32611-8105.
Phone (352) 294-2308. Fax (352) 392-8357.}
\and
Hongyan Hou\thanks{{\tt hongyan388@gmail.com},
        Chemical Engineering,
        Carnegie Mellon University,
        5000 Forbes Avenue, Pittsburgh, PA 15213.}
\and
Anil V. Rao\thanks{{\tt anilvrao@ufl.edu},
        http://www.mae.ufl.edu/rao
        Department of Mechanical and Aerospace Engineering,
        P.O. Box 116250, Gainesville, FL 32611-6250.
        Phone:(352) 392-0961. Fax:(352) 392-7303}
}

\begin{document}
\maketitle

\begin{abstract}
Estimates are obtained for the
Lebesgue constants associated with
the Gauss quadrature points on $(-1, +1)$ augmented by the point $-1$
and with the Radau quadrature points on either $(-1, +1]$ or $[-1, +1)$.
It is shown that the Lebesgue constants are $O(\sqrt{N})$,
where $N$ is the number of quadrature points.
These point sets arise in the estimation of the
residual associated with recently developed orthogonal collocation
schemes for optimal control problems.
For problems with smooth solutions,
the estimates for the Lebesgue constants can imply an exponential decay
of the residual in the collocated problem as a function of the number of
quadrature points.
\end{abstract}

\begin{keywords}
Lebesgue constants, Gauss quadrature, Radau quadrature, 
collocation methods
\end{keywords}

\pagestyle{myheadings} \thispagestyle{plain}
\markboth{W. W. HAGER, H. HOU, AND A. V. RAO}
{LEBESGUE CONSTANTS}

\section{Introduction}\label{introduction}
%
Recently, in \cite{ DarbyHagerRao11,DarbyHagerRao10, FrancolinHagerRao13,
GargHagerRao11b, GargHagerRao11a, GargHagerRao10a, PattersonHagerRao14},
a class of methods was developed for solving optimal control problems
using collocation at either Gauss or Radau quadrature points.
In \cite{HagerHouRao15b} and \cite{HagerHouRao15c} an
exponential convergence rate is established for these schemes.
The analysis is based on a bound for the inverse of a linearized operator
associated with the discretized problem, and an
estimate for the residual one gets when substituting the solution to the
continuous problem into the discretized problem.
This paper focuses on the estimation of the residual.
We show that the residual in the sup-norm is bounded by the sup-norm distance
between the derivative of the solution to the continuous problem and
the derivative of the interpolant of the solution.
By Markov's inequality $\cite{Markov1916}$,
this distance can be bounded in terms of the Lebesgue
constant for the point set and the error in best polynomial approximation.
A classic result of Jackson \cite{jackson} gives an estimate for
the error in best approximation.
The Lebesgue constant that we need to analyze corresponds to the
roots of a Jacobi polynomial on $(-1, +1)$
augmented by either $\tau = +1$ or $\tau = -1$.
The effects of the added endpoints were analyzed by
V\'{e}rtesi in \cite{Vertesi81}.
For either the Gauss quadrature points
on $(-1, +1)$ augmented by $\tau = +1$ or the Radau quadrature points on
$(-1, +1]$ or on $[-1, +1)$, the bound given in \cite[Thm. 2.1]{Vertesi81}
for the Lebesgue constants is $O(\log (N) \sqrt{N})$,
where $N$ is the number of quadrature points.
We sharpen this bound to $O(\sqrt{N})$.

To motivate the relevance of the Lebesgue constant to collocation methods,
let us consider the scalar first-order differential equation
\begin{equation} \label{de}
\dot{x}(\tau)=f\left(x(\tau)\right), \quad \tau \in [-1, +1],
\quad x(-1) = x_0,
\end{equation}
where $f : \mathbb{R}\rightarrow\mathbb{R}$.
%
In a collocation scheme for (\ref{de}),
the solution $x$ to the differential equation
(\ref{de}) is approximated by a polynomial $x$
that is required to satisfy the differential
equation at the collocation points.
Let us consider a scheme based on collocation at the Gauss quadrature
points $-1 < \tau_1 < \tau_2 < \ldots < \tau_N < +1$, the roots of the
Legendre polynomial of degree $N$.
In addition, we introduce the noncollocated point $\tau_0 = -1$.
The discretized problem is to find $x \in \C{P}_{N}$,
the space of polynomials of degree at most $N$, such that
\begin{equation}\label{collocated}
\dot{x}(\tau_k) = f(x(\tau_k)), \quad 1 \le k \le N,
\quad x(-1) = x_0.
\end{equation}
A polynomial of degree at most $N$ is uniquely specified by
$N+1$ parameters such as its coefficients.
The $N$ collocation equations and the boundary condition in (\ref{collocated})
yield $N+1$ equations for the polynomial.

The convergence of a solution of the collocated problem (\ref{collocated})
to a solution of the continuous problem (\ref{de})
ultimately depends on how accurately a polynomial interpolant of a
continuous solution satisfies the discrete equations (\ref{collocated}).
The Lagrange interpolation polynomials for the point set
$\{\tau_0, \tau_1, \ldots , \tau_N\}$ are defined by
\begin{equation}\label{lag}
L_i(\tau)=\prod_{\substack{j=0\\ j\neq i}}^N\frac{\tau-\tau_j}
{\tau_i-\tau_j}, \quad 0 \le i \le N.
\end{equation}
The interpolant $x^N$ of a solution $x$ to (\ref{de}) is given by
\[
x^N (\tau) = \sum_{j=0}^N x (\tau_j) L_j(\tau).
\]
%
The residual in (\ref{collocated})
associated with a solution of (\ref{de}) is the vector with components
\begin{equation}\label{res}
r_0 = x^N(-1) - x_0, \quad r_k = \dot{x}^N(\tau_k) - f(x^N(\tau_k)), \quad
1 \le k \le N.
\end{equation}
For the Gauss scheme,
$r_0 = 0$ since $x$ satisfies the boundary condition in (\ref{de}).
The potentially nonzero components of the residual are $r_k$, $1 \le k \le N$.

As we show in Section~\ref{residual}, the residual can be bounded
in terms of a Lebesgue constant and the error in best approximation for $x$ and
its derivative.
The Lebesgue constant $\Lambda_N$ relative to the point set
$\{\tau_0, \tau_1, \ldots , \tau_N\}$ is defined by
\begin{equation}\label{ln1}
\Lambda_N=\max \left\{
\sum_{j=0}^N\left|L_j(\tau)\right|: \tau\in[-1,1] \right\} .
\end{equation}
%
The article \cite{Brutman97} of Brutman gives a comprehensive survey on the
analysis of Lebesgue constants, while the book \cite{Mastroianni08}
of Mastroianni and Milovanovi\'{c} covers more recent results.

The paper is organized as follows.
In Section~\ref{residual}, we show how the Lebesgue constant enters
into the residual associated with the discretized problem (\ref{collocated}).
Section~\ref{szego} summarizes results of Szeg\H{o} used in the analysis.
Section~\ref{gauss+} analyzes the Lebesgue constant for the
Gauss quadrature points augmented by $\tau = -1$,
while Section~\ref{radau+} analyzes Radau quadrature points.
Finally, Section~\ref{tight} examines the tightness of the estimates
for the Lebesgue constants.

{\bf Notation.}
$\mathcal{P}_N$ denotes the space of polynomials of degree at most $N$
and $\|\cdot\|$ denotes the sup-norm on the interval $[-1, +1]$.
The Jacobi polynomial $P_N^{(\alpha, \beta)}(\tau)$,
$N \ge 1$, is an $N$-th degree polynomial, and for fixed $\alpha > -1$ and
$\beta > -1$, the polynomials are orthogonal on the interval $[-1, +1]$
relative to the weight function $(1-\tau)^\alpha(1+\tau)^\beta$.
$P_N$ stands for the Jacobi polynomial $P_N^{(0,0)}$, or equivalently,
the Legendre polynomial of degree $N$.
\section{Analysis of the residual}
\label{residual}
%
As discussed in the introduction,
a key step in the convergence analysis of collocation schemes
is the estimation of the residual defined in (\ref{res}).
The convergence of a discrete solution to the
solution of the continuous problem ultimately depends on
how quickly the residual approaches 0 as $N$ tends to infinity;
for example, see Theorem~3.1 in \cite{DontchevHager97},
Proposition~5.1 in \cite{Hager99c}, or Theorem~2.1 in \cite{Hager02b}.
Since a solution $x$ of (\ref{de}) satisfies the differential equation
on the interval $[-1, +1]$, it follows that
$\dot{x}(\tau_k) = f(x(\tau_k))$, $1 \le k \le N$.
Hence, the potentially nonzero components of the residual can be expressed
$r_k = \dot{x}^N (\tau_k) - \dot{x}(\tau_k)$, $1 \le k \le N$.
In other words, the size of the residual depends on the difference between
the derivative of the interpolating polynomial at the collocation
points and the derivative of the continuous solution at the collocation points.
Hence, let us consider the general problem of estimating the
difference between the derivative of an interpolating polynomial on the
point set $\tau_0 < \tau_1 < \ldots < \tau_N$ contained in $[-1, +1]$
and the derivative of the original function.
\smallskip
\begin{proposition}\label{L1}
If $x$ is continuously differentiable on $[-1, +1]$, then
\begin{eqnarray}
\left\|\dot{x}-\dot{x}^N\right\|
&\le& \left(1+2N^2\right) 
\inf_{q \in \mathcal{P}_{N}}\left\|\dot{x}-\dot{q}\right\| \nonumber \\
&& \quad + N^2(1+\Lambda_N)
\inf_{p \in \mathcal{P}_{N}}\left\|x-p\right\|
\label{diffy}
\end{eqnarray}
where $x^N \in \C{P}_N$ satisfies $x^N(\tau_k) = x(\tau_k)$,
$0 \le k \le N$, and $\Lambda_N$ is the Lebesgue constant relative to
the point set $\{ \tau_0, \tau_1, \ldots, \tau_N \}$.
\end{proposition}
\begin{proof}
Given $p \in \mathcal{P}_N$, the triangle inequality gives
\begin{equation}
\left\|\dot{x}-\dot{x}^N\right\|\leq \|\dot{x}-\dot{p}\|+\left\|\dot{p}
-\dot{x}^N\right\|.\label{dify}
\end{equation}
By Markov's inequality $\cite{Markov1916}$, we have
\begin{eqnarray}
\left\|\dot{p}-\dot{x}^N\right\|
&\leq& N^2 \left\|p-x^N\right\|=N^2 \left\|\sum_{i=0}^N(p(\tau_i)-x(\tau_i))
L_i(\tau)\right\|\nonumber \\
&\leq & N^2 \Lambda_N\max_{0\leq i\leq N}|p(\tau_i)-x(\tau_i)|
\le N^2\Lambda_N \|p-x\|. \label{qminusy}
\end{eqnarray}
Let $q \in \C{P}_{N}$ with $q(-1) = x(-1)$.
Again, by the triangle and Markov inequalities, we have
\begin{eqnarray}
\|\dot{x}-\dot{p}\| &\le& \|\dot{x} - \dot{q} \| + \|\dot{q} - \dot{p}\| \le
\|\dot{x} - \dot{q} \| + N^2 \|q - p\| \nonumber \\
&\le&
\|\dot{x} - \dot{q} \| + N^2 (\|q - x\| + \|x - p\|). \label{h71}
\end{eqnarray}
By the fundamental theorem of calculus,
\begin{equation}\notag
\left|q(t)-x(t)\right|=\left|\int_{-1}^{t}
\left(\dot{q}(s)-\dot{x}(s)\right)ds\right|\leq \int_{-1}^{t}
\left|\dot{q}(s)-\dot{x}(s)\right|ds\leq 2\|\dot{q}-\dot{x}\|.
\end{equation}
We combine this with (\ref{h71}) to obtain
\begin{equation}\label{h72}
\|\dot{x}-\dot{p}\| \le (1 + 2N^2) \|\dot{x} - \dot{q} \| + N^2 \|x - p\| .
\end{equation}
To complete the proof, combine (\ref{dify}), (\ref{qminusy}), and (\ref{h72})
and exploit the fact that
\[
\left\{\dot{q}: q(-1) = x(-1), \;\; q \in \C{P}_N \right\} =
\left\{\dot{q}: q \in \C{P}_N \right\}.
\]
\end{proof}

An estimate for the right side of \eqref{diffy} follows from results
on best uniform approximation by polynomials, which
originate from work of Jackson \cite{jackson}.
For example, the following result employs an estimate from Rivlin's
book \cite{Rivlin1969}.
\begin{lemma}\label{L2}
If $x$ has $m$ derivatives on $[-1, +1]$ and $N > m$, then
\begin{equation}\label{jackson}
\inf_{p\in \mathcal{P}_N}\|x-p\|\leq
\left( \frac{12}{m+1} \right) \left( \frac{6e}{N} \right)^m
\|x^{(m)}\|,
\end{equation}
where $x^{(m)}$ denotes the $m$-th derivative of $x$.
\end{lemma}
\begin{proof}
It is shown in \cite[Thm. 1.5]{Rivlin1969} that
\begin{equation}\label{yp}
\inf_{p\in \mathcal{P}_N}\left\|x-p\right\|\leq
\left( \frac{6}{m+1} \right) \left( \frac{6e}{N} \right)^m
\omega_m \left(\frac{1}{N-m}\right),
\end{equation}
where $\omega_m$ is the modulus of continuity of $x^{(m)}$. 
By the definition of the modulus of continuity, we have
\[
\omega_m\left(\frac{1}{N-m}\right)=\sup\left\{\left|x^{(m)}(\tau_1)
-x^{(m)}(\tau_2)\right|: {\tau_1, \tau_2 \in[-1,1], |\tau_1-\tau_2|
\leq 
\frac{1}{N-m}}\right\}.
\]
Since
\[
|x^{(m)}(\tau_1)-x^{(m)}(\tau_2) |\leq 2 
\|x^{(m)}\| ,
\]
(\ref{jackson}) follows from (\ref{yp}).
\end{proof}

If $\Lambda_N = O(N)$ and $m \ge 4$, then
Proposition~\ref{L1} and Lemma~\ref{L2} imply that the components
of the residual approach zero as $N$ tends to infinity.
Moreover, if $x$ is infinitely differentiable and
there exists a constant $c$ such that $\|x^{(m)}\| \le c^m$,
then we take $m = N-1$ in Lemma~\ref{L2} to obtain
\[
\inf_{p\in \mathcal{P}_N}\|x-p\|\leq
\left( \frac{2}{ec} \right) \left( \frac{6ec}{N} \right)^N.
\]
Hence, the convergence is extremely fast due to the $1/N^N$ factor.

\section{Some results of Szeg\H{o}}
\label{szego}
We now summarize several results developed by Szeg\H{o} in \cite{Szego1939}
for Jacobi polynomials that are used in the analysis.
The page and equation numbers that follow refer to the 2003 edition
of Szeg\H{o}'s book published by the American Mathematical Society.
First, at the bottom of page 338, Szeg\H{o} makes the following observation:
\smallskip
\begin{theorem}\label{jacobi}
The Lebesgue constant for the roots of the Jacobi polynomial 
$P_N^{(\alpha, \beta)}(\tau)$ is $O(N^{0.5+\gamma})$ 
if $\gamma := \max(\alpha, \beta) > -1/2$,
while it is $O(\log N)$ if $\gamma \le-1/2$.
\end{theorem}
\smallskip

For the Gauss quadrature points, $\alpha = \beta = 0$, $\gamma = 0$,
and $\Lambda_N = O(\sqrt{N})$.
The result that we state as Theorem~\ref{jacobi}
is based on a number of additional properties of Jacobi polynomials
which are useful in our analysis.
The following identity is a direct consequence of the Rodrigues formula
\cite[p. 67]{Szego1939} for $P_N^{(\alpha,\beta)}$.
\smallskip
\begin{proposition}\label{flip}
For any $\alpha$ and $\beta \in \mathbb{R}$, we have
\begin{equation}\label{eq8}
P_N^{(\alpha, \beta)}(\tau)=(-1)^NP_N^{(\beta, \alpha)}(-\tau)
\quad \mbox{for all } \tau \in [-1, +1].
\end{equation}
\end{proposition}
\smallskip

The following proposition provides some bounds for Jacobi polynomials.
\smallskip
\begin{proposition}\label{pro1}
For any $\alpha$ and $\beta \in \mathbb{R}$
and any fixed constant $c_1 > 0$,
we have
\[
P_N^{(\alpha,\beta)}(\cos\theta)=\left\{
\begin{array}{clcccl}
O\left(N^\alpha\right) &\mbox{if } \theta \in
[&0&,& c_1N^{-1} &],\\[.05in]
\theta^{-\alpha-0.5}O\left(N^{-1/2}\right)
&\mbox{if } \theta \in [ &c_1N^{-1} &, & \pi/2 &],\\[.05in]
(\pi-\theta)^{-\beta-0.5}O\left(N^{-1/2}\right)
&\mbox{if } \theta \in [&\pi/2&,& \pi- c_1N^{-1}&],\\[.05in]
O\left(N^\beta\right) &\mbox{if } \theta \in [&\pi- c_1N^{-1}&,& \pi&].
\end{array}
\right.
\]
\end{proposition}
\smallskip
\begin{proof}
The bounds for $\theta \in [0, cN^{-1}]$ and for
$\theta \in [cN^{-1}, \pi/2]$ appear in \cite[(7.32.5)]{Szego1939}.
If $\theta \in \left[\pi/2, \pi\right]$, then
$\pi-\theta \in \left[0, \pi/2 \right]$ and by \eqref{eq8},
\begin{equation}\label{h1}
P_N^{(\alpha, \beta)}(\cos \theta)=P_N^{(\alpha, \beta)}(-\cos(\pi- \theta))
=(-1)^NP_N^{(\beta, \alpha)}(\cos(\pi- \theta)).
\end{equation}
Hence, for $\theta \in [\pi/2, \pi]$,
the first two estimates in the proposition applied to the right
side of (\ref{h1}) yield the last two estimates.
\end{proof}

The next proposition provides an estimate for the derivative of a
Jacobi polynomial at a zero.
\smallskip
\begin{proposition}\label{pro2}
If $\alpha>-1$ and $\beta>-1$, then there exist constants
$\gamma_2 \ge \gamma_1 > 0$, depending only on $\alpha$ and $\beta$, such that
\[
\gamma_1 i^{-\beta - 1.5} N^{\beta + 2} \le
\left|\dot{P}_N^{(\alpha, \beta)}(\tau_i)\right| \le
\gamma_2 i^{-\beta - 1.5} N^{\beta + 2}
\]
whenever $\tau_i \le 0$ where
$\tau_1 < \tau_2 < \ldots < \tau_N$ are the zeros of $P_N^{(\alpha, \beta)}$
(the smallest zero is indexed first).
Moreover, if $\theta_i \in [0, \pi]$ is defined by
$\cos \theta_i = \tau_i$, then there exist constants
$\gamma_4 \ge \gamma_3 > 0$, depending only on $\alpha$ and $\beta$, such that
\begin{equation}\label{h9}
\gamma_3 \sqrt{N} (\pi - \theta_i)^{-\beta - 1.5} \le
\left|\dot{P}_N^{(\alpha, \beta)}(\tau_i)\right| \le
\gamma_4 \sqrt{N} (\pi -\theta_i)^{-\beta - 1.5}
\end{equation}
whenever $\theta_i \in [\pi/2, \pi]$.
\end{proposition}
\smallskip
\begin{proof}
In \cite[(8.9.2)]{Szego1939}, it is shown that there exist
$\gamma_2 \ge \gamma_1 > 0$, depending only on $\alpha$ and $\beta$, such that
\begin{equation}\label{h7}
\gamma_1 i^{-\beta - 1.5} N^{\beta + 2} \le
\left|\dot{P}_N^{(\beta, \alpha)}(\sigma_i)\right| \le
\gamma_2 i^{-\beta - 1.5} N^{\beta + 2}
\end{equation}
whenever $\sigma_i \ge 0$ where
$\sigma_1 > \sigma_2 > \ldots > \sigma_N$ are the zeros of
$P_N^{(\beta, \alpha)}$ (the largest zero is indexed first).
By Proposition~\ref{flip}, $\tau_i$ is a zero of $P_N^{(\alpha,\beta)}$
if and only if $-\tau_i$ is a zero of $P_N^{(\beta,\alpha)}$.
Hence, the zeros of $P_N^{(\beta,\alpha)}$ are
$-\tau_1 > -\tau_{2} > \ldots > -\tau_N$.
Moreover,
\begin{equation}\label{h7.5}
\dot{P}_N^{(\alpha,\beta)}(\tau) = \pm
\dot{P}_N^{(\beta,\alpha)}(-\tau).
\end{equation}
The bound given in the proposition for
$|\dot{P}_N^{(\alpha,\beta)}(\tau_i)|$ with $\tau_i \le 0$ is exactly the
bound (\ref{h7}) for
$|\dot{P}_N^{(\beta,\alpha)}(\sigma_i)|$ with $\sigma_i \ge 0$.

It is shown in \cite[(8.9.7)]{Szego1939}, that there exist constants
$\gamma_4 \ge \gamma_3 > 0$, depending only on $\alpha$ and $\beta$, such that
\begin{equation}\label{h8}
\gamma_3 \sqrt{N} \phi_i^{-\beta - 1.5} \le
\left|\dot{P}_N^{(\beta, \alpha)}(\sigma_i)\right| \le
\gamma_4 \sqrt{N} \phi_i^{-\beta - 1.5}
\end{equation}
whenever $\phi_i \in [0, \pi/2]$ where $\cos \phi_i = \sigma_i$.
Since $\cos \phi_i = \sigma_i = -\tau_i = \cos (\pi - \theta_i)$,
it follows that $\phi_i = \pi - \theta_i$, and
(\ref{h7.5}) and (\ref{h8}) yield (\ref{h9}).
\end{proof}

%
\section{Lebesgue constant for Gauss quadrature points augmented by $-1$}
\label{gauss+}

In this section we estimate the Lebesgue constant for
the Gauss quadrature points augmented by $\tau_0 = -1$.
Due to the symmetry of the Gauss quadrature points, the same
estimate holds when the Gauss quadrature points are augmented by $+1$
instead of $-1$.
The Gauss quadrature points are the zeros of the Jacobi polynomial
$P_N^{(0, 0)}(\tau)$, which is abbreviated as $P_N(\tau)$.
By Theorem~\ref{jacobi}, the Lebesgue constant for the Gauss
quadrature points themselves is $O(\sqrt{N})$.
The effect of adding the point $\tau_0 = -1$ to the Gauss quadrature
points is not immediately clear due to the new factor $(1 + \tau_i)$
in the denominator of the Lagrange polynomials;
this factor can approach 0 since roots of $P_N$
approach $-1$ as $N$ tends to infinity.
Nonetheless, with a careful grouping of terms,
Szeg\H{o}'s bound in Theorem~\ref{jacobi}
for the Gauss quadrature points can be extended to handle the new
point $\tau_0 = -1$.
\smallskip
\begin{theorem}\label{gausstheom}
The Lebesgue constant for the point set consisting of the Gauss
quadrature points $-1 < \tau_1 < \tau_2 < \ldots < \tau_N < +1$
$($the zeros of $P_N)$ augmented with $\tau_0 = -1$ is $O(\sqrt{N})$.
\end{theorem}
\smallskip

\begin{proof}
Define
\[l(\tau)=(\tau-\tau_1)(\tau-\tau_2)\dots (\tau-\tau_N),
\quad \mbox{and}\quad  L(\tau)=(\tau+1)l(\tau).
\]
The derivative of $L(\tau)$ at $\tau_i$ is
\[
\dot{L}(\tau_i)=l(\tau_i)+(\tau_i+1)\dot{l}(\tau_i)=\left\{
\begin{array}{cl}\displaystyle
l(-1), & i = 0, \\[.1in]
(\tau_i+1)\dot{l}(\tau_i),  &i> 0.
\end{array}
\right.
\]
Hence, the Lagrange polynomials $L_i(\tau)$ associated with the
point set $\{\tau_0 , \tau_1, \ldots, \tau_N\}$ can be expressed as
\begin{equation}\label{Li}
L_i(\tau)=\frac{L(\tau)}{\dot{L}(\tau_i)(\tau-\tau_i)}=\left\{
\begin{array}{cl}
l(\tau)/l(-1), &i=0, \\[.1in]
\displaystyle\frac{L(\tau)}{(\tau_i+1)\dot{l}(\tau_i)(\tau-\tau_i)}, 
& i> 0.
\end{array}
\right.
\end{equation}
Since $P_N$ is a multiple of $l$ (it has the same zeros), it follows that
\[
L_i(\tau)=\left\{
\begin{array}{cl}
P_N(\tau)/P_N(-1), &i=0,\\[.1in]
\displaystyle
\frac{(\tau+1)P_N(\tau)}{(\tau_i+1)\dot{P}_N(\tau_i)(\tau-\tau_i)}, 
&i > 0.
\end{array}
\right.
\]
By \cite[(7.21.1)]{Szego1939},
$|P_N(\tau)| \le 1$ for all $\tau \in [-1, +1]$, and by
\cite[(4.1.4)]{Szego1939}, $|P_N(-1)| = (-1)^N$.
We conclude that $|L_0 (\tau)| \le 1$ for all $\tau \in [-1, +1]$.
Hence, the proof is complete if
\begin{equation}\label{h3}
\max \left\{ \sum_{i=1}^N|L_i(\tau)| : \tau \in[-1, 1] \right\} = O(\sqrt{N}) .
\end{equation}

For any $\tau \in [-1, +1]$, the integers $i \in [1, N]$ are partitioned
into the four disjoint sets
\begin{eqnarray*}
\C{I}_1 &=& \{ i \in [1,N]: \tau_i \ge 0 \}, \\
\C{I}_2 &=& \{ i \in [1,N]: -1 < \tau_i < 0, \; \tau_i > \tau \}, \\
\C{I}_3 &=& \{ i \in [1,N]: -1 < \tau_i < 0, \; \tau_i \le \tau, \;
\tau - \tau_i \le \tau_i + 1 \}, \\
\C{I}_4 &=& \{ i \in [1,N]: -1 < \tau_i < 0, \; \tau_i \le \tau, \;
\tau - \tau_i > \tau_i + 1 \}.
\end{eqnarray*}
Let $\C{I}_{123}$ denote $\C{I}_1 \cup \C{I}_2 \cup \C{I}_3$.
Observe that for any $i \in \C{I}_{123}$ and $\tau \in [-1, +1]$,
$(\tau+1)/(\tau_i + 1) \le 2$.
Consequently, for all $i \in \C{I}_{123}$,
\[
|L_i (\tau)| =
\left| \frac{(\tau+1)P_N(\tau)}{(\tau_i+1)\dot{P}_N(\tau_i)(\tau-\tau_i)}
\right| \le
\frac{2|P_N(\tau)|}{|\dot{P}_N(\tau_i)(\tau-\tau_i)|} .
\]
This bound together with Theorem~\ref{jacobi} imply that
\[
\sum_{i \in \C{I}_{123}} |L_i(\tau)| \le
\sum_{i \in \C{I}_{123}}
\frac{2|P_N(\tau)|}{|\dot{P}_N(\tau_i)(\tau-\tau_i)|}  \le
2 \sum_{i=1}^N
\frac{|P_N(\tau)|}{|\dot{P}_N(\tau_i)(\tau-\tau_i)|} = O(\sqrt{N})
\]
since the terms in the final sum are the Lagrange
polynomials for the Gauss quadrature points.
To complete the proof, we need to analyze the terms in (\ref{h3})
associated with the indices in $\C{I}_4$.
These terms are more difficult to analyze since $\tau_i + 1$
in the denominator of $L_i$ could approach 0 while $\tau +1$ in
the numerator remains bounded away from 0.

For $i \in \C{I}_4$, we have
\[
\tau + 1 = (\tau - \tau_i) + (\tau_i + 1) \le 2 (\tau - \tau_i)
\]
since $\tau - \tau_i > \tau_i + 1$.
Hence,
\[
|L_i (\tau)| \le  \frac{2|P_N(\tau)|}{|(\tau_i + 1) \dot{P}_N (\tau_i)|} \le
\frac{2}{|(\tau_i + 1) \dot{P}_N (\tau_i)|}
\]
since $|P_N(\tau)| \le 1$ for all $\tau \in [-1, +1]$
by \cite[(7.21.1)]{Szego1939}.
It follows that
\begin{equation}\label{h6}
\sum_{i \in \C{I}_{4}} |L_i(\tau)| \le
\sum_{i \in \C{I}_{4}}
\frac{2}{|(\tau_i + 1) \dot{P}_N (\tau_i)|} \le
\sum_{-1 < \tau_i < 0 }
\frac{2}{|(\tau_i + 1) \dot{P}_N (\tau_i)|} .
\end{equation}

Given $\theta \in [\pi/2, \pi]$, define $\phi = \pi - \theta$.
Observe that
\[
\left|\frac{\phi^2}{1+\cos \theta}\right|
=\frac{\phi^2}{2\cos^2(\theta/2)}
=\frac{2(\phi/2)^2}{\sin^2 (\phi/2)}
\leq \max_{x\in [0, \pi/4]}
\frac{2x^2}{\sin^2 x} =\frac{\pi^2}{4}.
\]
Hence, for $\theta \in [\pi/2, \pi]$, we have
\begin{equation}\label{h5}
1 + \cos \theta \ge \left( \frac{4}{\pi^2} \right) \phi^2 =
\frac{4}{\pi^2} (\pi - \theta)^2 .
\end{equation}
By the bounds \cite[(6.21.5)]{Szego1939} for the roots of the
Jacobi polynomial $P_N^{(\alpha, \beta)}$ when
$\alpha$ and $\beta \in [-0.5, +0.5]$, it follows that
\begin{equation}\label{*}
\left(\frac{2i-1}{2N+1}\right) \pi \leq \pi-\theta_i
\leq \left(\frac{2i}{2N+1}\right) \pi, \quad
1 \le i \le N,
\end{equation}
where $\cos \theta_i = \tau_i$.
This implies the lower bound
\begin{equation}\label{h4}
\pi - \theta_i \ge
\left(\frac{2i-1}{2N+1}\right) \pi \ge
\left( \frac{i}{3N} \right)\pi > \frac{i}{N} .
\end{equation}
We combine (\ref{h5}) and (\ref{h4}) to obtain
\begin{equation}\label{eq3}
1+\tau_i \ge \frac{4}{\pi^2}(\pi-\theta_i)^2\geq\frac{4}{\pi^2}
\left(\frac{i}{N}\right)^2.
\end{equation}

By Proposition~\ref{pro2},
\[
|\dot{P}_N(\cos \theta_i )| \ge \gamma_1 i^{-1.5} N^2.
\]
This lower bound for the derivative and the lower bound (\ref{eq3}) for
the root imply that
\[
\frac{1}{(1+\tau_i)|\dot{P}_N(\tau_{i})|} \le
\left( \frac{\pi^2}{4 \gamma_1} \right) i^{-1/2} .
\]
Hence, we obtain the following bound for the $\C{I}_4$ sum in (\ref{h6}):
\[
\sum_{-1<\tau_i<0}\frac{2}{(1+\tau_i)|\dot{P}_N(\tau_{i})|} \le
\left( \frac{\pi^2}{2 \gamma_1} \right)
\sum_{i = 1}^N i^{-1/2} \le
\left( \frac{\pi^2}{2 \gamma_1} \right)
\int_0^N i^{-1/2} di = O(\sqrt{N}) .
\]
This bound inserted in (\ref{h6}) completes the proof.
\end{proof}

\section{Lebesgue constants for the Radau quadrature points}
\label{radau+}

Next, we estimate the Lebesgue constant for the Radau quadrature scheme.
There are two versions of the Radau quadrature points depending on whether
$\tau_1 = -1$ or $\tau_N = +1$.
Since these two schemes have quadrature points that are the
negatives of one another, the Lebesgue constants are the same.
The analysis is carried out for the case $\tau_N = +1$.
In this case, the Radau quadrature points are the $N-1$ roots of
$P_{N-1}^{(1,0)}$ augmented by $\tau_N = 1$.
Szeg\H{o} shows that the Lebesgue constant for the roots of
$P_{N-1}^{(1,0)}$ is $O(N^{3/2})$. 
We show that when the quadrature point $\tau_N = 1$ is included,
the Lebesgue constant drops to $O(\sqrt{N})$. 

The analysis requires an estimate for the location of the zeros of
$P_{N-1}^{(1,0)}$.
Our estimate is based on some relatively recent results on
interlacing properties for the zeros of Jacobi polynomials obtained by
Driver, Jordaan, and Mbuyi in \cite{DriverJordaanMbuyi2008}.
Let $\tau_i'$ and $\tau_i''$, $i\geq 1$, be zeros of
$P_{N-1}$ and $P_{N}$ respectively, arranged in increasing order.
Applying \cite[Thm. 2.2]{DriverJordaanMbuyi2008}, we have
\[
\tau_i'' < \tau_i < \tau_{i}' ,
\]
$i = 1, 2, \ldots, N-1$, where
$-1 < \tau_1 < \tau_2 < \ldots < \tau_{N-1} < +1$ are the zeros of
$P_{N-1}^{(1,0)}$.
Let $\theta_i \in [0, \pi]$ be defined by $\cos \theta_i = \tau_i$.
By the estimate (\ref{*}) for the zeros of $P_N$, it follows that
the zeros of $P_{N-1}^{(1,0)}$ have the property that
\begin{equation}\label{zeros}
\left( \frac{2i-1}{2N-1} \right) \pi < \theta_{N-i} <
\left( \frac{2(i+1)}{2N+1} \right) \pi, \quad
1 \le i \le N-1.
\end{equation}
When $i$ is replaced by $N-i$, these bounds become
\begin{equation}\label{zeros*}
\left( \frac{2i-1}{2N+1} \right) \pi < \pi - \theta_{i} <
\left( \frac{2i}{2N-1} \right) \pi, \quad
1 \le i \le N-1.
\end{equation}
Together, (\ref{zeros}) and (\ref{zeros*}) imply that
\begin{equation}\label{phibounds}
\pi - \theta_{i} > i/N \quad \mbox{and} \quad \theta_{N-i} > i/N,
\quad 1 \le i \le N-1;
\end{equation}
moreover, taking into account both the upper and lower bounds, we have
\begin{eqnarray}
\theta_i - \theta_{i+1} &<& \left( \frac{4(i+N)+2N+1}{4N^2 - 1}\right) \pi
\le \left( \frac{10N - 7}{4N^2 - 1} \right) \pi \nonumber \\
&<& \left( \frac{5(2N - 1)}{4N^2 - 1} \right) \pi < \frac{2.5\pi}{N},
\quad 1 \le i \le N-2.
\label{separation}
\end{eqnarray}
Thus, the interlacing properties for the zeros leads to explicit
bounds for the separation of the zeros; for comparison,
Theorem~8.9.1 in \cite{Szego1939} yields $\theta_i - \theta_{i+1} = O(1)/N$,
while (\ref{separation}) yields an explicit constant $2.5\pi$.
These estimates for the zeros of $P_{N-1}^{(1,0)}$ are used
to derive the following result.
\smallskip
\begin{theorem}\label{radau}
The Lebesgue constant for the Radau quadrature points
\[
-1 < \tau_1 < \tau_2 < \ldots < \tau_N = 1
\]
$($the zeros of $P_{N-1}^{(1,0)}$ augmented by $\tau_N = +1)$ is $O(\sqrt{N})$.
\end{theorem}
\smallskip
\begin{proof}
The Lagrange interpolating polynomials $R_i$, $1 \le i \le N$,
associated with the Radau quadrature points are given by
\[
R_i(\tau)= \left( \frac{1-\tau}{1-\tau_i} \right)
\prod_{\substack{j=1\\ j\neq i}}^{N-1}\frac{\tau-\tau_j}
{\tau_i-\tau_j}, \quad 1 \le i \le N-1, \quad
R_N(\tau) =
\prod_{\substack{j=1}}^{N-1}\frac{\tau-\tau_j}
{1-\tau_j}.
\]
Similar to (\ref{Li}), the $R_i$ can be expressed
\begin{equation}\label{Ri}
R_i(\tau)=\left\{
\begin{array}{cl}
\displaystyle\frac{(1-\tau)P_{N-1}^{(1,0)}(\tau)}
{(1-\tau_i)\dot{P}_{N-1}^{(1,0)}(\tau_i)(\tau-\tau_i)}, 
&i < N, \\[.20in]
\displaystyle{\frac{P_{N-1}^{(1,0)}(\tau)}{P_{N-1}^{(1,0)}(1)}}.
&i=N.\\
\end{array}
\right.
\end{equation}

By \cite[(4.1.1)]{Szego1939} and \cite[(7.32.2)]{Szego1939}, we have
\begin{equation}\label{h22}
P_{N-1}^{(1,0)}(1)= N \quad \mbox{and} \quad
|P_{N-1}^{(1,0)}(\tau)|\le N \mbox{ for all } \tau \in [-1, +1] .
\end{equation}
We conclude that $|R_N (\tau)| \le 1$ for all $\tau \in [-1, +1]$.
Hence, the proof is complete if
\begin{equation}\label{h10}
\max \left\{ \sum_{i=1}^{N-1}|R_i(\tau)| : \tau \in [-1, +1] \right\}
=O(\sqrt{N}) .
\end{equation}

Let $\delta > 0$ be a small constant.
Technically, any $\delta$ satisfying $0 < \delta < 1/2$ is
small enough for the analysis.
Szeg\H{o} establishes the following bounds when analyzing the
Lebesgue constants associated with the roots of Jacobi polynomials:
\begin{equation}\label{radaulebesgue}
\sum_{i = 1}^N \left| \frac{P_{N}^{(1,0)}(\tau)}
{\dot{P}_{N}^{(1,0)}(\tau_i)(\tau-\tau_i)} \right| =
\left\{
\begin{array}{ll}
O(\sqrt{N}) & \mbox{if } \tau \in [-1, \delta-1], \\
O(\log N) & \mbox{if } \tau \in [\delta-1 , 1 - \delta], \\
O(N^{3/2}) & \mbox{if } \tau \in [1 - \delta, 1].
\end{array} \right.
\end{equation}
Szeg\H{o} considers the general Jacobi polynomials
$P_N^{(\alpha, \beta)}$ on pages 336--338 of \cite{Szego1939},
while here we only state the results
corresponding to $\alpha = 1$ and $\beta = 0$.

We first show that (\ref{h10}) holds when $\tau \in [-1, 1-\delta]$.
Observe that $(1 - \tau)/(1-\tau_i) \le 4/\delta$
when $\tau_i \le 1 - \delta/2$ and $\tau \in [-1, +1]$.
It follows from (\ref{radaulebesgue}) that
\begin{eqnarray}
\sum_{\tau_i \le 1-\delta/2} |R_i(\tau)| &\le& \left( \frac{4}{\delta} \right)
\sum_{\tau_i \le 1-\delta/2}
\left| \frac{P_{N-1}^{(1,0)}(\tau)}
{\dot{P}_{N-1}^{(1,0)}(\tau_i)(\tau-\tau_i)} \right| \nonumber \\
&=& \left\{ \begin{array}{ll}
O(\sqrt{N}), & \tau \in [-1, \delta-1], \\
O(\log N), & \tau \in [\delta-1, 1 - \delta] .
\end{array} \right. \label{h11}
\end{eqnarray}

When $\tau_i > 1-\delta/2$ and $\tau \in [-1, +1-\delta]$, we have
$|\tau - \tau_i| \ge \delta/2$; hence,
\begin{equation}\label{h12}
\sum_{1 > \tau_i > 1-\delta/2} |R_i(\tau)| \le \left( \frac{4}{\delta} \right)
\sum_{1 > \tau_i > 1-\delta/2}
\left| \frac{P_{N-1}^{(1,0)}(\tau)}
{(\tau_i - 1)\dot{P}_{N-1}^{(1,0)}(\tau_i)} \right| .
\end{equation}
We have the following bounds for the factors on the right side of (\ref{h12}):
\begin{itemize}
\item[(a)]
By Proposition~\ref{pro1},
$|P_{N-1}^{(1,0)} (\tau)| = O(1)$ if $\tau \in [-1, \delta - 1]$ and
$|P_{N-1}^{(1,0)} (\tau)| = O(N^{-1/2})$ if $\tau \in [\delta - 1, 1-\delta]$.
\item[(b)]
By (\ref{h8}),
$|\dot{P}_{N-1}^{(1, 0)}(\tau_i)| \ge
\gamma_3 \theta_i^{-5/2} \sqrt{N-1}$, where $\cos \theta_i = \tau_i \ge 0$.
\item[(c)]
By a Taylor expansion around $\theta = 0$,
\begin{equation}\label{1-cos}
\theta^2/4 \le 1 - \cos \theta \le \theta^2/2, \quad \theta \in [0, \pi/2].
\end{equation}
\end{itemize}
By (b) and the lower bound in (c) at $\theta = \theta_i$, we have
\begin{equation}\label{h100}
(1-\tau_i) |\dot{P}_{N-1}^{(1,0)}(\tau_i)| \ge
0.25 \gamma_3 \theta_i^{-1/2} \sqrt{N-1} .
\end{equation}
We combine this with (a) and (\ref{h12}) to obtain
\[
\sum_{1 > \tau_i > 1-\delta/2} |R_i(\tau)| =
\left\{ \begin{array}{lll}
O(N^{-1/2})\displaystyle\sum_{i = 1}^N \sqrt{\theta_i} &= O(\sqrt{N}),
& \tau \in [-1, \delta - 1], \\
O(N^{-1})\displaystyle\sum_{i = 1}^N \sqrt{\theta_i} &= O(1),
& \tau \in [\delta-1, 1-\delta] ,
\end{array} \right.
\]
since $\theta_i \in [0, \pi]$.
This establishes (\ref{h11}) for all $\tau \in [-1, 1-\delta]$.

To complete the proof of (\ref{h10}), we need to consider
$\tau \in (1-\delta, 1]$.
The analysis becomes more complex since
Szeg\H{o}'s estimate (\ref{radaulebesgue}) is $O(N^{3/2})$ in this region,
while we are trying to establish a much smaller bound in (\ref{h10});
in fact, the bound in this region is $O(\log N)$ as we will show.
For the numerator of $R_i (\tau)$ and
$\tau \in (1-\delta, 1]$,
Proposition~\ref{pro1} and (\ref{1-cos}) yield
\begin{eqnarray}
(1-\tau) |P_{N-1}^{(1,0)} (\tau)| &=&
(1-\cos \theta)|P_{N-1}^{(1,0)}(\cos \theta)| =
\left\{ \begin{array}{ll}
\theta^2 O(N), & \theta \in [0, 1/N], \\
\theta^{1/2} O(N^{-1/2}), & \theta \in [1/N, \pi/2],
\end{array} \right. \nonumber \\
&=&
O(N^{-1/2}). \label{h14}
\end{eqnarray}
Given $\tau \in (1-\delta, 1]$,
let us first focus on those $i$ in (\ref{h10}) for which
$|\tau-\tau_i| \ge \delta$.
In this case, $\tau_i \le 1-\delta$ or
$1-\tau_i \ge \delta$, and (\ref{h14}) gives
\begin{eqnarray}
\sum_{|\tau-\tau_i|\ge\delta}|R_i(\tau)| &=&
\sum_{|\tau-\tau_i|\ge\delta}
\left| \frac{(\tau-1)P_{N-1}^{(1,0)}(\tau)}
{(\tau_i-1)\dot{P}_{N-1}^{(1,0)}(\tau_i)(\tau-\tau_i)} \right| \nonumber \\
&\le& \frac{O(N^{-1/2})}{\delta^2} \sum_{|\tau-\tau_i|\ge\delta} 
\frac{1}{|\dot{P}_{N-1}^{(1,0)} (\tau_i)|}. \label{h50}
\end{eqnarray}
The lower bounds (\ref{h9}) and (\ref{h8}) imply that
\begin{equation}\label{h51}
\sum_{|\tau-\tau_i|\ge\delta}|R_i(\tau)| =
O(N^{-1}) \sum_{\tau_i \ge 0} \theta_i^{5/2}
+
O(N^{-1}) \sum_{\tau_i < 0} |\pi - \theta_i|^{3/2} = O(1),
\end{equation}
since the terms in the sums are uniformly bounded and there are at
most $N$ terms.

The next step in the proof of (\ref{h10}) for $\tau \in (1-\delta, 1]$
is to consider those terms corresponding to $|\tau-\tau_i|<\delta$.
Since $\delta$ is small, it follows that both $\tau$ and $\tau_i$ are
near 1, and consequently, $\theta$ and $\theta_i$
are small and nonnegative,
where $\cos \theta = \tau$ and $\cos \theta_i = \tau_i$.
In particular, $0 \le \theta_i \le \pi/2$.
In this case where $\tau_i$ is near $\tau$,
it is important to take into account the fact that
$\tau - \tau_i$ is a divisor of the numerator $P_{N-1}^{(1,0)}(\tau)$.
To begin, we combine the lower bound in (\ref{h8}) and the bounds in
(\ref{1-cos}) to obtain
\begin{equation}\label{h19}
\frac{(1-\tau)}
{(1-\tau_i)|\dot{P}_{N-1}^{(1,0)}(\tau_i)|} \le
\frac{2\theta^2}{\theta_i^2 (\gamma_3 \sqrt{N}) \theta_i^{-5/2}} =
O(N^{-1/2}) \theta^2\sqrt{\theta_i} .
\end{equation}
It follows from (\ref{Ri}) that
\begin{equation}\label{h15}
|R_i(\tau)|= O(N^{-1/2}) \theta^2\sqrt{\theta_i}
\left| \frac{P_{N-1}^{(1,0)}(\tau)}{\tau-\tau_i} \right| .
\end{equation}
%

The mean value theorem and the formula
\cite[(4.21.7)]{Szego1939} for the derivative of 
$P_{N-1}^{(\alpha, \beta)}(\tau)$ in terms of 
$P_{N-2}^{(\alpha+1, \beta+1)}(\tau)$ gives the identity
\begin{equation} \label{h17}
\left|\frac{P_{N-1}^{(1,0)}(\tau)}{\tau-\tau_i}\right|
= \left|\frac{P_{N-1}^{(1,0)}(\tau)-P_{N-1}^{(1,0)}
(\tau_i)}{\tau-\tau_i}\right|
=\left(\frac{N+1}{2}\right)\left|P_{N-2}^{(2,1)}(\cos\eta_i)\right|,
\end{equation}
where $\eta_i$ lies between $\theta$ and $\theta_i$. 
Together, (\ref{h15}) and (\ref{h17}) imply that
\begin{equation}\label{h20}
|R_i(\tau)| = O(N^{1/2}) \theta^2 \sqrt{\theta_i}
\left|P_{N-2}^{(2,1)}(\cos\eta_i)\right|.
\end{equation}

The estimate (\ref{h20}) is useful when $\tau_i$ is near $\tau$.
When $\tau_i$ is not near $\tau$, we proceed as follows.
Use the identity
\[
\cos\alpha-\cos\beta
=-2\sin\displaystyle{\frac{(\alpha+\beta)}{2}}
\sin\displaystyle{\frac{(\alpha-\beta)}{2}},
\]
to deduce that
\begin{equation}\label{h70}
|\tau - \tau_i| = |\cos \theta - \cos \theta_i| \ge
\frac{2}{\pi^2} \left| \theta^2 - \theta_i^2 \right|
\end{equation}
when $|\theta + \theta_i| \le \pi$, which is satisfied since
both $\theta$ and $\theta_i$ are near 0.
Exploiting this inequality in (\ref{h15}) yields
\begin{equation}\label{h21}
|R_i(\tau)|= O(N^{-1/2}) \theta^2\sqrt{\theta_i}
\left| \frac{P_{N-1}^{(1,0)}(\tau)}{ \theta^2 - \theta_i^2} \right| .
\end{equation}

Recall, that we now need to analyze the interval $\tau \in [1-\delta, 1]$
and those $i$ for which $|\tau-\tau_i| < \delta$.
Our analysis works with the variable $\theta \in [0, \pi/2]$,
where $\cos \theta = \tau$.
The interval $\theta \in [0, \pi/2]$ corresponds to
$\tau \in [0,1]$ which covers the target interval $[1-\delta, 1]$ when
$\delta$ is small.
By \cite[(7.32.2)]{Szego1939}, we have
\[
|P_{N-2}^{(2,1)}(\cos\eta_i)| \le N(N-1)/2 .
\]
If $\theta \in [0, c/N]$, where $c$ is a fixed constant independent of $N$,
then it follows from
(\ref{h20}) that $|R_i(\tau)| = O(N^{1/2}) \sqrt{\theta_i}$.
Moreover, if $\theta_i \le 2 \theta \le 2c/N$, then
$|R_i (\tau)| = O(1)$.
By the bounds (\ref{phibounds}),
the number of roots that
satisfy $\theta_{N-i} \le 2c/N$ is at most $2c$, independent of $N$.
On the other hand, if $\theta_i > 2 \theta$, then $\theta< \theta_i/2$ and
\[
\left|\theta_i^2-\theta^2\right|=\theta_i^2-\theta^2\geq 
(3/4) \theta_i^2 .
\]
With this substitution in (\ref{h21}), we have
\[
|R_i(\tau)|= O(N^{-1/2}) \theta^2\theta_i^{-3/2}
\left| {P_{N-1}^{(1,0)}(\tau)} \right| .
\]
By (\ref{h22}), $| P_{N-1}^{(1,0)}(\tau)| \le N$.
Hence, if $\theta \in [0, c/N]$, then by (\ref{phibounds}), we have
\begin{eqnarray*}
\sum_{|\tau-\tau_i| < \delta} |R_i(\tau)| &=&
O(N^{-3/2}) \sum_{|\tau-\tau_i| < \delta} \theta_i^{-3/2} =
O(N^{-3/2}) \sum_{i=1}^{N-1} \theta_i^{-3/2} \\
&=& O(N^{-3/2}) \sum_{i=1}^{N-1} \theta_{N-i}^{-3/2} =
O(1) \sum_{i=1}^{N-1} i^{-3/2} = O(1),
\end{eqnarray*}
for all $\theta \in [0, c/N]$.

Finally, suppose that $\theta \in [c/N, \pi/2]$.
By (\ref{separation}) the separation between adjacent zeros
$\theta_i$ and $\theta_{i+1}$ is at most $2.5\pi/N$.
Hence, if $\theta_i$ is within $k$ zeros of $\theta$, then
$\eta_i \ge \theta - \gamma N^{-1}$, $\gamma = 2.5\pi k$.
Here $k \ge 2$ is an arbitrary fixed integer.
By Proposition~\ref{pro1}, we have
\[
\left|P_{N-2}^{(2,1)}(\cos\eta_i)\right| =
(\theta- \gamma N^{-1})^{-5/2}O(N^{-1/2}) .
\]
Choose $c > 2\gamma$.
If $\theta \in [c/N, \pi/2]$,
then $\theta/2 \ge c/(2N) \ge \gamma/N$.
Hence, $\theta- \gamma /N \ge \theta/2$ and
\[
\left|P_{N-2}^{(2,1)}(\cos\eta_i)\right| =
(\theta/2)^{-5/2}O(N^{-1/2}) =
\theta^{-5/2}O(N^{-1/2}) .
\]
Combine this with (\ref{h20}) to obtain
\[
|R_i(\tau)| = O(1) \sqrt{\theta_i/\theta} .
\]
when $\theta \in [c/N, \pi/2]$ and $\theta_i$ is within $k$
zeros of $\theta$.
If $\theta_i \le \theta$, then $R_i (\tau) = O(1)$.
If $\theta_i > \theta$ and $\theta_i$ is within $k$ zeros of $\theta$,
then $\theta_i - \theta \le \gamma/N$, and
\[
\theta_i/\theta \le (\theta + \gamma/N)/\theta \le 1 + \gamma/c
\]
when $\theta \in [cN, \pi/2]$.
Thus $|R_i (\tau)| = O(1)$ when $\theta \in [cN, \pi/2]$ and
$\theta_i$ is within $k$ zeros of $\theta$.

This analysis of $R_i$ when $\theta_i$ is close to $\theta$ needs to
be complemented with an analysis of $R_i$ when $\theta_i$ is not
close to $\theta$ and $\theta \in [c/N, \pi/2]$.
For $\theta$ in this interval, Proposition~\ref{pro1} yields
$|P_{N-1}^{(1,0)} (\cos \theta)| = \theta^{-3/2}O(N^{-1/2})$.
By (\ref{h21}), we have
\begin{equation}\label{h23}
|R_i (\tau)| = O(N^{-1}) \frac{\sqrt{\theta} \sqrt{\theta_i}}
{|\theta^2 - \theta_i^2|} .
\end{equation}
If $\theta \ge 2\theta_i$, then
$\theta^2 - \theta_i^2 \ge (3/4)\theta^2$ and
\[
|R_i (\tau)| = O(N^{-1}) \theta^{-3/2} \theta_i^{1/2} .
\]
By (\ref{zeros*}), we have
\[
|R_{N-i} (\tau)| = O((N\theta)^{-3/2}) \sqrt{i+1} .
\]
Recall that we are focusing on those $i$ for which $\theta_{N-i} \le \theta/2$.
The lower bound $\theta_{N-i} \ge i/N$ from (\ref{phibounds}) implies that
$i \le N\theta/2$ whenever $\theta_{N-i} \le \theta/2$.
Hence, the set of $i$
satisfying $i \le N\theta$ is a superset of the $i$ that we need to consider,
and we have
\begin{eqnarray*}
\sum_{\theta_i \le \theta/2} |R_i (\tau)| &=&
\sum_{\theta_{N-i} \le \theta/2} |R_{N-i} (\tau)| =
O((N\theta)^{-3/2}) \sum_{i \le N\theta} \sqrt{i+1} \\
&=& O((N\theta)^{-3/2}) (N\theta + 1)^{3/2} = O(1) .
\end{eqnarray*}

On the other hand, if $\theta < 2 \theta_i$, then we have
\[
\frac{\sqrt{\theta} \sqrt{\theta_i}}
{|\theta^2 - \theta_i^2|} =
\frac{\sqrt{\theta} \sqrt{\theta_i}}
{|(\theta - \theta_i)(\theta + \theta_i)|} \le
\frac{\sqrt{\theta} \sqrt{\theta_i}}
{|(\theta - \theta_i)\theta_i|} \le
\frac{\sqrt{2}}
{|\theta - \theta_i|} .
\]
Combine this with (\ref{h23}) to obtain
\[
|R_i (\tau)| = \frac{O(1)}{|N\theta - N\theta_i|} .
\]
Earlier we showed that
$|R_i (\tau)| = O(1)$ for those $i$ where the associated $\theta_i$
is within $k$ zeros of $\theta$.
When $\theta_i$ is more than $k$ zeros away from $\theta$,
we exploit the estimate (\ref{zeros}) for the zeros to deduce that
$|N\theta - N\theta_i|$ behaves like an arithmetic sequence of natural numbers.
Hence, the sum of the $|R_i (\tau)|$ over these natural numbers,
where we avoid the singularity, is bounded by a multiple of $\log N$.
This completes the proof.
\end{proof}

\section{Tightness of estimates}\label{numerical}
\label{tight}
At the bottom of page 110 in \cite{Vertesi81}, V\'{e}rtesi states
some lower bounds for the Lebesgue function.
In the case of the Gauss quadrature points augmented by $\tau_{N+1} = +1$
and the Radau quadrature points with $\tau_N = +1$, the associated
Lebesgue function is of order $\sqrt{N}$ at
$\tau = (\tau_1 + \tau_{2})/2$,
the midpoint between the two smallest quadrature points.
It follows that the $O(\sqrt{N})$ estimates for the Lebesgue constant are tight.
To study the tightness of the estimates,
the Lebesgue constants were evaluated numerically and
fit by curves of the form $a \sqrt{N} + b$, $10 \le N \le 100$
(see Figures~\ref{graphgauss}--\ref{graphradau}).
A fast and accurate method for evaluating the Gauss quadrature points,
which could be extended to the Radau quadrature points,
is given by Hale and Townsend in \cite{HaleTownsend13}.
Figure~\ref{graphgauss}--\ref{graphradau} indicate that
a curve of the form $a \sqrt{N} + b$ is a good fit to the Lebesgue constant.
Another Lebesgue constant which enters into the analysis of
the Radau collocation schemes studied in \cite{HagerHouRao15c} is the
Lebesgue constant for the Radau quadrature points on $(-1, +1]$
augmented by $\tau_0 = -1$.
As given by V\'{e}rtesi in \cite[Thm. 2.1]{Vertesi81}, the
Lebesgue constant is $O(\log n)$.
Trefethen \cite{Trefethen13} points out that the Lebesgue constant
on any point set has the lower bound
\[
\Lambda_N \ge \left( \frac{2}{\pi} \right) \log (N) + 0.52125\ldots ,
\]
due to Erd\H{o}s \cite{Erdos61} and Brutman \cite{Brutman78}.
For comparison, Figure~\ref{graphradau-1} plots this lower bound
along with the computed Lebesgue constant.
When the number of interpolation points range between 10 and 100,
the Lebesgue constant for the Radau quadrature
points augmented by the point $-1$ differs from the smallest possible Lebesgue
constant by between 0.70 and 0.84.
\begin{figure}
\centering
\includegraphics[scale=.4]{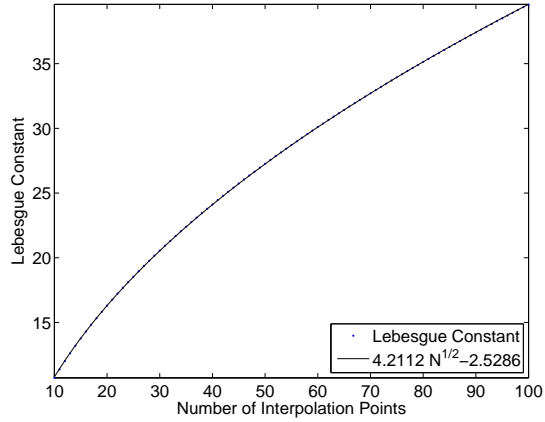}
\caption{Least squares approximation to the Lebesgue constant for 
the point set corresponding to the Gauss quadrature points augmented by $-1$ 
using curves of the form $a\sqrt{N}+b$}
\label{graphgauss}
\end{figure}
\begin{figure}
\centering
\includegraphics[scale=.4]{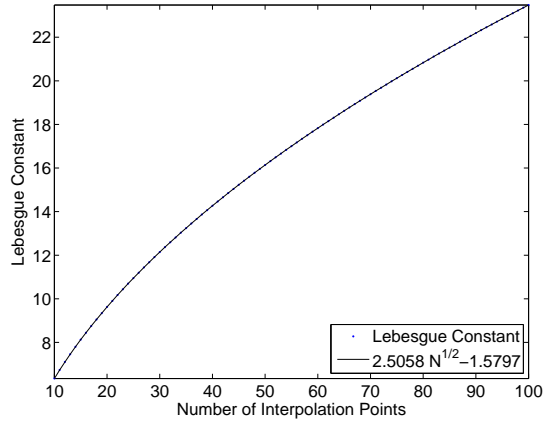}
\caption{Least squares approximation to the Lebesgue constant for 
the point set corresponding to the Radau quadrature points using curves of the 
form $a\sqrt{N}+b$}
\label{graphradau}
\end{figure}
\begin{figure}
\centering
\includegraphics[scale=.4]{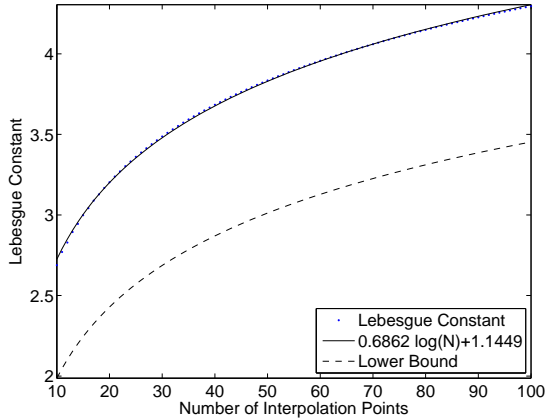}
\caption{Least squares approximation to the Lebesgue constant for 
the point set corresponding to the Radau quadrature points on
$(-1, +1]$ augmented by $-1$ using curves of the form $a\log N +b$}
\label{graphradau-1}
\end{figure}
\section{Conclusions}
\label{conclusions}
In Gauss and Radau collocation methods for unconstrained control problems
\cite{HagerHouRao15b, HagerHouRao15c},
the error in the solution to the discrete problem is bounded by the
residual for the solution to the continuous problem inserted in the
discrete equations.
In Section~\ref{residual}, we observe that the residual in the sup-norm
is bounded by the distance between the derivative of the continuous
solution interpolant and the derivative of the continuous solution.
Proposition~\ref{L1} bounds this distance in terms of the error in
best approximation and the Lebesgue constant for the point set.
We show that the Lebesgue constant for the point sets associated with
the Gauss and Radau collocation methods is $O(\sqrt{N})$, and
by the plots of Section~\ref{numerical}, the Lebesgue constants are
closely fit by curves of the form $a\sqrt{N}+b$.

\section*{Acknowledgments}
Special thanks to Lloyd N. Trefethen for pointing out Brutman's paper
\cite{Brutman97} and for providing a copy when we had trouble locating
the journal.
Also, we thank a reviewer for pointing out the book \cite{Mastroianni08}
which contains newer results as well as additional references.
\bibliographystyle{siam}
\bibliography{library}
\end{document}